\documentclass[12pt]{amsart}
\usepackage{amsmath,amssymb,amsfonts,amsthm,amsopn}
\usepackage{graphicx,tikz}
\usepackage[all]{xy}
\usepackage{amsmath}
\usepackage{multirow, longtable, makecell, caption, array}
\usepackage{amssymb}
\usepackage{mathtools}
\usepackage{pb-diagram}
\usepackage{cellspace}

\theoremstyle{plain}
\newtheorem{theorem}{Theorem}
\newtheorem{proposition}{Proposition}
\newtheorem{lemma}{Lemma}
\newtheorem{corollary}{Corollary}
\theoremstyle{definition}
\newtheorem{definition}{Definition}
\theoremstyle{remark}
\newtheorem{remark}{Remark}

\def\Bnorm#1{{ \| #1 \|_\Bsp }}
\def\NSPB#1{{\left (#1,\| \ebbes \|_{#1} \right )}}

\newcommand{\Asp}{{\boldsymbol A}}     
\newcommand{\AspN}{(\Asp, \, \|\ebbes\|_\Asp)}     
\newcommand{\ebbes}{\mbox{$\,\cdot\,$}}     
\newcommand{\Bsp}{{\boldsymbol B}}     
\newcommand{\BspN}{(\Bsp, \, \|\ebbes\|_\Bsp)}     
\newcommand\CG{{\Csp(\cG)}}     
\newcommand{\Csp}{{\boldsymbol C}}     
\newcommand{\cG}{\mathscr{G}}     
\newcommand{\COG}{{\COsp(G)}}     
\newcommand{\COsp}{{\Csp_{\negthinspace 0}}}     
\newcommand\COGN{{\big( \COsp(G), \, \|\ebbes\|_\infty \big)}}     
\newcommand{\COPG}{{\COPsp(G)}}     
\newcommand{\COPsp}{{\Csp'_{\negthinspace 0}}}     
\newcommand\COPGN{{\big( \COPsp(G), \, \|\ebbes\|_\COPsp \big)}}     
\newcommand{\CORd}{{\COsp(\Rst^d)}}     
\newcommand{\Rst}{{\mathbb R}}     
\newcommand{\CbG}{{\Cbsp(G)}}     
\newcommand{\Cbsp}{{\Csp_{\negthinspace b}}}     
\newcommand{\CcG}{{\Ccsp(G)}}     
\newcommand{\Ccsp}{{\Csp_{\negthinspace c}}}     
\newcommand{\CcRd}{{\Ccsp(\Rst^d)}}     
\newcommand{\DPhi}{{\operatorname {D}}_\Phi}     
\newcommand\DPhimu{{\DPhi \mu}}     
\newcommand{\DPsi}{{\operatorname{D}_\Psi}}     
\newcommand\DPsimu{{\DPsi \mu }}     
\newcommand{\DRd}{{\Dcsp(\Rst^d)}}     
\newcommand{\Dcsp}{{\boldsymbol{\mathcal D}}}     
\newcommand{\FF}{{\mathcal{F}}}     
\newcommand\FLi{{\mathcal F}{\negthinspace \Lisp}}     
\newcommand{\Lisp}{{\Lsp^1}}     
\newcommand{\FLiRd}{{ \FLi(\Rdst) }}     
\newcommand{\Rdst}{{{\Rst^d}}}     
\newcommand\FLiRdN{\big( \FLiRd, \, \|\ebbes\|_{\FLisp} \big)}     
\newcommand{\FLisp}{{ {\mathcal F} \negthinspace \Lisp}}     
\newcommand\FLispN{\big( \FLiRd, \, \|\ebbes\|_{\FLisp} \big)}     
\newcommand\FLpsp{{\FF \negthinspace \Lpsp}}     
\newcommand{\Lpsp}{{\Lsp^p}}     
\newcommand{\FT}{{\operatorname{{\mathcal F}}}}     
\newcommand\HsRd{{{\mathcal H}_s (\Rdst)}}
\newcommand{\LiG}{{\Lisp(G)}}     
\newcommand{\LiGN}{\big( \LiG, \, \|\ebbes\|_1 \big)}     
\newcommand{\LiRd}{{\Lisp \nth (\Rst^d)}}     
\newcommand\nth{\negthinspace}     
\newcommand{\LiRdN}{\big( \LiRd, \, \|\ebbes\|_1 \big)}     
\newcommand\Liloc{ {\Lsp^1_{ \negthinspace \it{loc}} } }     
\newcommand{\Lsp}{{\boldsymbol L}}     
\newcommand{\Linsp}{{\Lsp^\infty}}     
\newcommand{\LpG}{{\Lpsp(G)}}     
\newcommand{\LpGN}{\big( \LpG, \, \|\ebbes\|_p \big)}     
\newcommand{\LtG}{{\Ltsp(G)}}     
\newcommand{\Ltsp}{{\Lsp^2}}     
\newcommand{\LtR}{{\Ltsp(\Rst)}}     
\newcommand{\LtRN}{\big( \LtR, \, \|\ebbes\|_2 \big)}     
\newcommand{\LtRd}{{\Ltsp(\Rst^d)}}     
\newcommand{\LtRdN}{\big( \LtRd, \, \|\ebbes\|_2 \big)}     
\newcommand{\MG}{{\Mbsp(\cG)}}     
\newcommand{\Mbsp}{{\Msp_{\negthinspace b}}}     
\newcommand\MGN{{(\Mbsp(G), \| \ebbes \|_\Mbsp )}}     
\newcommand{\MRd}{{ \Msp{(\Rdst)}}}     
\newcommand{\Msp}{{\boldsymbol M}}     
\newcommand{\MbG}{{\Msp(G)}}     
\newcommand\MRdN{{(\Msp(\Rst^d), \| \ebbes \|_\Msp )}}     
\newcommand{\MdG}{{\Mdsp(G)}}
\newcommand{\Mdsp}{{\Msp_{\negthinspace d}}}     

\newcommand\MdGN{{(\MdG, \| \ebbes \|_\Mbsp )}}
\newcommand{\Nst}{{\mathbb N}}     
\newcommand\Phifam{{ \Phi = (\phi_j)_{j \in J} }}     
\newcommand\Psifam{{ \Psi = (\psi_i)_{i \in I} }}     
\newcommand{\Rdsth}{{\widehat{\Rst}^d}}     
\newcommand{\SOG}{{\SOsp(G)}}     
\newcommand{\SOsp}{{\Ssp_{\negthinspace 0}}}     
\newcommand{\SOGN}{\big( \SOG, \|\ebbes\|_\SOsp \big)}     
\newcommand{\SOGTrG}{{ (\SOsp,\Ltsp,\SOPsp)(G) }}     
\newcommand{\SOPsp}{{\Ssp_{\negthinspace 0}'}}     
\newcommand{\SOPG}{{\SOPsp(G)}}     
\newcommand{\Ssp}{{\boldsymbol S}}     
\newcommand{\SORd}{{\SOsp(\Rst^d)}}     
\newcommand{\ScPsp}{{\Scsp'}}     
\newcommand{\Scsp}{{\boldsymbol{\mathcal S}}}     
\newcommand{\ScRd}{{\Scsp(\Rst^d)}}     
\newcommand{\Sp}{\operatorname{Sp}}     
\newcommand{\SpPsi}{\Sp_\Psi \nth}     
\newcommand\SpPsif{{\SpPsi \nth f}}     
\newcommand{\Strho}{{\operatorname{St}_{\negthinspace \rho}}}     
\newcommand\TFRd{{{\Rdst\,{\times}\,\Rdsth}}}     
\newcommand\UUG{{ U \in \UnbG}}     
\newcommand\UnbG{{ \mathcal{U}(e)}}     
\newcommand{\WCOliG}{\Wsp(\COsp,\lisp)(G)}     
\newcommand{\Wsp}{{\boldsymbol W}}     
\newcommand{\lisp}{{\lsp^1}}     
\newcommand{\WCOliGN}{{\big( \WCOliG, \, \|\ebbes\|_\Wsp \big)}}     
\newcommand{\WCOliRd}{\Wsp(\COsp,\lisp)(\Rdst) }     
\newcommand{\WCOlisp}{\Wsp(\COsp,\lisp)}     
\newcommand\WFLili{{\Wsp(\FLi,\lisp)}}     
\newcommand\WFLiliRd{{\Wsp(\FLi,\lisp)(\Rdst)}}     
\newcommand\WFLiliRdN{{(\WFLiliRd, \|\ebbes\|_{\WFLili})}}     
\newcommand\WLilin{{\Wsp(\Lisp,\linsp)}}     
\newcommand{\linsp}{{\lsp^\infty}}     
\newcommand\WLinfli{{ \Wsp(\Linsp,\lisp) }}     
\newcommand\WMlinf{ \Wsp(\Msp,\linsp)}     
\newcommand{\WspN}{(\Wsp, \, \|\ebbes\|_\Wsp)}     
\newcommand{\WG}{{\Wsp(G)}}
\newcommand{\Zdst}{{\Zst^d}}     
\newcommand{\Zst}{{\mathbb Z}}     
\newcommand{\alinI}{{{\alpha{\in}I}}}     
\newcommand\boxcar{{{\bf 1}_{[-1/2,1/2]}}}     
\newcommand{\cD}{\mathscr{D}}     
\newcommand\cGd{{\widehat{\cG}}}     
\newcommand\capS{{\operatorname{cap}_S}}     
\newcommand{\checkm}{{^\checkmark}}     
\newcommand{\delo}{\delta > 0}     
\newcommand{\delx}{\delta_x}     
\newcommand{\diam}{\operatorname{diam}}     
\newcommand\epso{{ \varepsilon > 0 }}     
\newcommand\fBN{{ \|f\|_\Bsp}}     
\newcommand\fchk{{f^\checkmark}}     
\newcommand\finB{{f \in \Bsp}}     
\newcommand\hatf{{\widehat{f}}}     
\newcommand\hkr{\hookrightarrow}     
\newcommand{\infnorm}[1]{{\lVert #1 \rVert_\infty}}     
\newcommand{\intG}{\int_{G}}     
\newcommand{\inv}{^{-1}}     
\newcommand{\lainLa}{{\lambda{\in}\Lambda}}     
\newcommand\lato{{ \lambda = (t,\omega) }}     
\newcommand\limPsitoz{{\lim_{|\Psi| \to 0}}}     
\newcommand\limal{{\lim_{\alpha \to \infty} \, }}     
\newcommand\limalinf{{ \lim_{\alpha \to \infty}}}     
\newcommand{\linfnorm}[1]{{\lVert #1 \rVert_{\infty}}}     
\newcommand{\linorm}[1]{{\lVert #1 \rVert_1}}     
\newcommand{\lsp}{{\boldsymbol\ell}}     
\newcommand\lqsp{{ \lsp^q}}     
\newcommand\muMN{ \| \mu \|_\Msp }     
\newcommand\mual{{\mu_\alpha}}     
\newcommand{\mualiI}{(\mu_\alpha)_\alinI}     
\newcommand{\nnth}{{ \negthinspace \: \negthinspace }}     
\def\normta#1#2{{  \| {#1}   \|_{#2} \, }}
\newcommand{\ofp}[1]{{\slp{#1}\srp}}     
\newcommand{\slp}{{{\raise 0.5pt \hbox{\footnotesize $($}}}}     
\newcommand{\srp}{{{\raise 0.5pt \hbox{\footnotesize $)$}}}}     
\newcommand\psii{{\psi_i}}     

\def\pxi{{T_{x_i}p}}  
\newcommand\rhobul{{\bullet_{\rho}}}     
\newcommand\sPsi{|\Psi|}     
\newcommand\sPsitoz{{\, \sPsi \to 0 }}     
\def\shahf{{\makebox[2.3ex][s]{$\sqcup$\hspace{-0.15em}\hfill $\sqcup$}\, }}
\newcommand\siScP{{ \sigma \in \ScPsp }}     
\newcommand\sumiF{{\sum_{i \in F}}}     
\newcommand{\sumiI}{\sum_{i\in I}}     
\newcommand\sumkinf{\sum_{k=1}^\infty}     
\newcommand{\sumniinf}{\sum_{n=1}^{\infty}}     
\newcommand\supiI{{\sup_{i \in I}}}     
\newcommand{\supp}{\operatorname{supp}}     
\newcommand\suth{{ \, | \, } }     
\newcommand\tblue{\textcolor{blue}}     
\newcommand\tred{\textcolor{red}}     
\newcommand{\vareps}{\varepsilon}     
\newcommand\veps{{\varepsilon}}     
\newcommand{\wdash}{{ w^* \negthinspace \mbox{-}}}     
\newcommand\wst{ w^{*} }     
\newcommand\wstd{{w^* \negthinspace -}}     
\newcommand\wstlim{{ \wdash \lim \, }}     
\newcommand\wstlimal{{\wstlim_{\alpha \to \infty} \, }}     
\newcommand\xalI{{(x_\alpha)_{\alpha \in I}}}     
\newcommand\xiiI{{(x_i)_{i \in I}}}     

\def\citeX{\cite}
\def\tred{}
\def\tblue{}

\def\cG{G} %
\def\cD{\mathscr{D}}
\def\cD{D}
\def\MG{{\Msp(\cG)}}  
\def\MGN{{(\MG, \| \ebbes\|_\Msp)}}   

\def\MbG{\MG}

\def\Mbnorm#1{{\|#1\|_\Msp}}

\def\citeX{\cite}

\begin{document}


\vspace{3mm}
\centerline{\large  Homogeneous Banach spaces as
Banach convolution modules over $\MG$}
\vspace{3mm}
\centerline{ \large Hans G. Feichtinger\footnote{Address: Fac. Math.,
 University Vienna, Oskar-Morgenstern-Platz 1, 1090 Wien, AUSTRIA;
\newline Acoustics Research Institute, Austrian Academy of Sciences, Wohllebengasse 12-14, 1040 Vienna.\newline
{\tt ORCID 0000-0002-9927-0742}
}
}
\vspace{4mm}

This paper is supposed to form a keystone towards a new and
alternative approach to Fourier Analysis over LCA (locally compact Abelian) groups $\cG$. Already in an earlier paper the author has shown
that one can introduce convolution and the Fourier-Stieltjes transform
on $\MGN$, the space of bounded measures (viewed as a space of linear functionals) in an elementary fashion over $\Rdst$.

Bounded uniform partitions of unity (BUPUs) are easily constructed
in the Euclidean setting (by dilation).
Moving on to  general LCA groups it becomes an interesting
challenge to find ways to construct arbitrary {\it fine}  BUPUs,
ideally without the use of structure theory, the existence of a Haar measure
and even Lebesgue integration.

This article provides such a construction and demonstrates
how it can be used in order to show that any so-called
{\it homogeneous  Banach space} $\BspN$ on $\cG$, such as
$\LpGN$, for $ 1 \leq p < \infty$, or the Fourier
Stieltjes algebra $ \FT \MbG$, and in particular any {\it Segal algebra}
is a {\it Banach convolution module} over $\MGN$ in a natural way.

Via the Haar measure we can then identify $\LiGN$ with the closure
(of the embedded version) of $\CcG$, the space of continuous functions
with compact support,  in $\MGN$, and show that these homogeneous
Banach spaces are {\it essential} $\LiG$-modules. Thus in particular
the approximate units act properly as one might expect and
converge strongly to the identity operator.

The approach is in the spirit
of Hans Reiter, avoiding the use of structure theory for LCA groups
and the usual techniques of vector-valued integration via duality.
The ultimate (still far)  goal of this approach is to provide a new and
elementary approach towards the (extended)
Fourier transform in the setting of
 the so-called {\it Banach Gelfand triple}
$\SOGTrG$, based on the Segal algebra $\SOG$. This direction will be pursued
in subsequent papers.

\vspace{6mm}

\section{Introduction} \label{secintro}

Let us begin with the observation that the usual approach to
Harmonic Analysis over locally compact Abelian (LCA) groups $\cG$
(see for example \citeX{re68}, \citeX{rest00}, \citeX{fo95}, \citeX{de02}) starts with a description of the Lebesgue space $\LiGN$, which
turns out to be a Banach algebra with respect to convolution. Based on the
description of the Fourier transform as an integral transform the traditional
approach continues with the  demonstration of the fact  that the
Fourier transform turns convolution into pointwise multiplication
(the so-called convolution theorem).
This result describes one of the crucial properties of the Fourier transform,
and the Lebesgue space appears to be very natural and best possible domain,
because it allows to describe the convolution product of two functions (more
precisely of equivalence classes of measurable functions) in the pointwise
(almost everywhere) sense, combined with the corresponding norm estimate
$$ \normta {f \ast g} {\Lisp} \leq \normta {f } {\Lisp} \normta { g} {\Lisp}, \quad f,g \in \LiG. $$
It is also plausible that $\LiGN$ is considered as the natural
domain for the Fourier transform, because for any {\it character} $\chi \in \cGd$ the integral
\begin{equation} \label{defFourier04}
 \hatf(\chi) = \intG  f(x) \overline{\chi(x)} dx
\end{equation}
exists in the Lebesgue sense (for one and then for any $\chi \in \cGd$)
if and only if $f \in \LiG$. In a similar way it appears as a natural
restriction to assume that $\hatf$ belongs to $\Lisp(\cGd)$ if one
wants to obtain $f(x)$ back (again via the usual integral formula
describing the inverse Fourier transform) from $\hatf$.
The range of the Fourier transform is denoted
by $\FLiRdN$. It is a Banach algebra with respect to pointwise
multiplication, hence called the {\it Fourier algebra}, with
respect to the norm $\normta{\hatf} {\FLisp} := \normta {f} {\Lisp}$.

While technically demanding, this approach based on measure theory
allows to formulate and answer interesting mathematical questions
(e.g. about the almost everywhere convergence of Fourier series),
but it does {\it not reveal the relevance of convolution for applications}.
The situation is different when
moving on to tempered distributions, which have become the key tool
for the treatment of PDEs. But in order to make use of these tools
it is required to first study to some extent the Schwartz space $\ScRd$,
a nuclear  Fr\'echet space with a countable system of seminorms involving
differentiation. For general LCA groups one can define the Schwartz-Bruhat space via structure theory, but it is even more complicated and very difficult to use.


Recalling the fact that engineers learn about the concept of  {\it convolution}
in their introductory courses on translation-invariant linear systems (TILS) this author has so far developed an approach to convolution (for bounded measures) which is based on the isometric one-to-one correspondence between linear functionals on $\COGN$ (we call them bounded measures and use the symbol
$\MGN$) and bounded linear operators commuting with translations.
Obviously the space  $\COGN$ of continuous, complex-valued functions
vanishing at infinity form a Banach space (even a pointwise algebra) if endowed with the sup-norm, and $\CcG$ (compactly supported functions) are dense in $\COGN$. It is also invariant under translations, defined as usual by
\begin{equation} \label{transldef04}
[T_z f](y) = f(y-z), \quad y,z \in \cG.
\end{equation}
Any such TILS can be identified with a moving average resp.\ a convolution operator by a uniquely determined bounded measure  $\mu \in \MGN = \COPGN$. This isometric identification allows to transfer the
{\it composition structure of linear operators} to the corresponding bounded
measures, and call it {\it convolution}. Of course this viewpoint is compatible with the usual approach (see \citeX{fo95}, p.46). As it turns out, that it is the unique $\wstd$continuous extension of the 
identification of translation operators $T_x$ with the corresponding Dirac
measures $\delta_x \in \MbG$. In this way $\COGN$ is a Banach module over
$\MGN$ with respect to convolution. Details are given in \citeX{fe17} (and in the Lecture Notes for the ETH course, see {\tt www.nuhag.eu/ETH20}).

The realization of this correspondence makes use of so-called BUPUs, i.e.
bounded uniform partitions of unity. They allow to decompose every $\mu \in \MG$ into an absolutely convergent sum of well localized measures, which
among others allows the extension of the action of $\mu \in \COPG$ to
all of $\CbG$, the continuous, bounded functions on $\cG$ (also endowed
with the sup-norm). In this way it is possible to define the Fourier-Stieltjes
transform of bounded measures and derive the convolution theorem before even
discussing the existence of a Haar measure or the necessary Lebesgue integration theory required in order to study everything in the $\Lisp$-context.

The goal of the present manuscript  is to provide an important step towards a description of the (generalized) Fourier transform over LCA groups along the lines of the approach described above.  This author is convinced that the appropriate setting is that of the Banach Gelfand Triple $\SOGTrG$, consisting of the Segal algebra $\SOG$, which can be defined on arbitrary LCA groups, its
dual space $\SOPG$, the space of so-called mild distributions, and in the
middle the Hilbert space $\LtG$ (defined as the completion of $\SOG$ with
respect to usual scalar product).

While such an approach can be realized easily in the context of $\cG = \Rdst$, the Euclidean setting, making use of special ingredients available
in this context, notably the existence of a Fourier invariant Gaussian function and dilation operators, which among others allow to create arbitrary fine
BUPUs in a natural fashion, it is not so obvious whether and how one can
obtain such BUPUs in the context of an abstract LCA group. Moreover, many important convolution relations make use of the fact that convolution operators induced by bounded measures act also boundedly on a large variety of Banach
spaces of functions over the group $\cG$, e.g.\ on the usual spaces $\LpGN$,
or the Fourier algebra $\FLisp(G)$ and (hence) on $\SOGN$. We will provide a relatively simple construction of such arbitrary fine BUPUs, avoiding the use of structure theory of LCA groups, and derive similar results making use of
these BUPUs.

The natural setting for the realization of such a general statement is the
setting of {\it homogeneous Banach spaces} (HBS) (in the sense of Y.~Katznelson), which are isometrically translation invariant by assumption. The family of {\it Segal algebras} (in the sense of H.~Reiter) is an interesting subfamily of this class of Banach spaces of locally integrable functions over $\cG$. The second main result of this paper will deal with
such Banach spaces and will demonstrate that any such HBS $\BspN$
is actually a Banach module over $\MGN$ (hence over $\LiGN$) 
with respect to convolution.

\vspace{2mm}
{\bf The paper is organized in the following way.}
First we discuss several
variations of the concept of a BUPU in Section \ref{secBUPUs}, a {\it bounded uniform partition of unity} and explain their mutual relationship. We also provide a few historical comments on their use in the literature.

In Section \ref{secCAPBUPU} the existence of arbitrary fine BUPUs is established as our first main result. 
Instead of the Haar measure we use a kind of coarse measurement of the
size of sets, called a {\it capacity} (with respect to a sufficiently small
reference set). This provides the basis for our key results,
without making use of the structure theory for LCA groups. 
Subsequently it is shown in 
Section   \ref{OpBUPUs} how to make use of such  BUPUs. 
In Section \ref{secWG} we also discuss various  characterizations of  the
{\it Wiener Algebra} $\Wsp = \WCOliG$ and its dual via BUPUs.

In Section \ref{secHBS} our second main result is shown:
any homogeneous Banach space (in the sense of Y.~Katznelson)
is a Banach   module over $\MGN$ with respect to convolution.
In fact, we formulate an even more general abstract approach
based on isometric, strongly continuous representations of the
group $\cG$ on  an arbitrary Banach space $\BspN$.
The   approach is based on the methods developed in \citeX{fe17} 
and makes use of a constructive way of approximating  
bounded measures by discrete measures in the $\wstd$sense. The technical
realization of this second main result is based on the completeness
of Banach spaces, which also implies that (bounded) {\it Cauchy nets}
are actually convergent in any Banach space. The necessary
background is described in the  Section \ref{secFA}. The approach
also permits to demonstrate that the $\wstd$convergence of bounded
and tight nets leads to strong operator convergence of the
corresponding convolution operators  (Thm. \ref{wsttightconv}). 

Only then the existence of the Haar measure is invoked in
order to define $\LiGN$ as a subspace of $\MGN$, namely as 
closure of $\CcG$. In this sense Section \ref{secHBS}
characterizes the usual {\it integrated
group representation} as the restriction of the established
module structure over $\MG$. In particular it is shown that  any
homogeneous Banach space is also an {\it essential}
Banach modules over $\LiGN$.

\section{Different Types of Uniform Partitions} \label{secBUPUs}

It is the purpose of this section to compare various notions of uniform
partitions of unity in the context of harmonic analysis over LCA groups.
It is easy to construct arbitrary fine BUPUs of a given degree of smoothness on $\Rst$ by just applying appropriate dilations
to the basis of B-splines of sufficiently high order (or even infinitely
differentiable)
which are obtained as translations along the integer lattice $\Zst$ of the
convolution powers of the indicator function $\boxcar$.
For B-splines of order $3$ (four-fold convolution power) one obtains
a Riesz basis for the cubic spline function in $\LtRN$.
Via tensor products the same can be achieved on $\Rdst$ for $d \geq 2$.

\begin{figure}
\label{BUPUqreg4.jpg}
\begin{center} \parbox[t]{13cm}{
\includegraphics
[width=15cm,height= 7cm]
{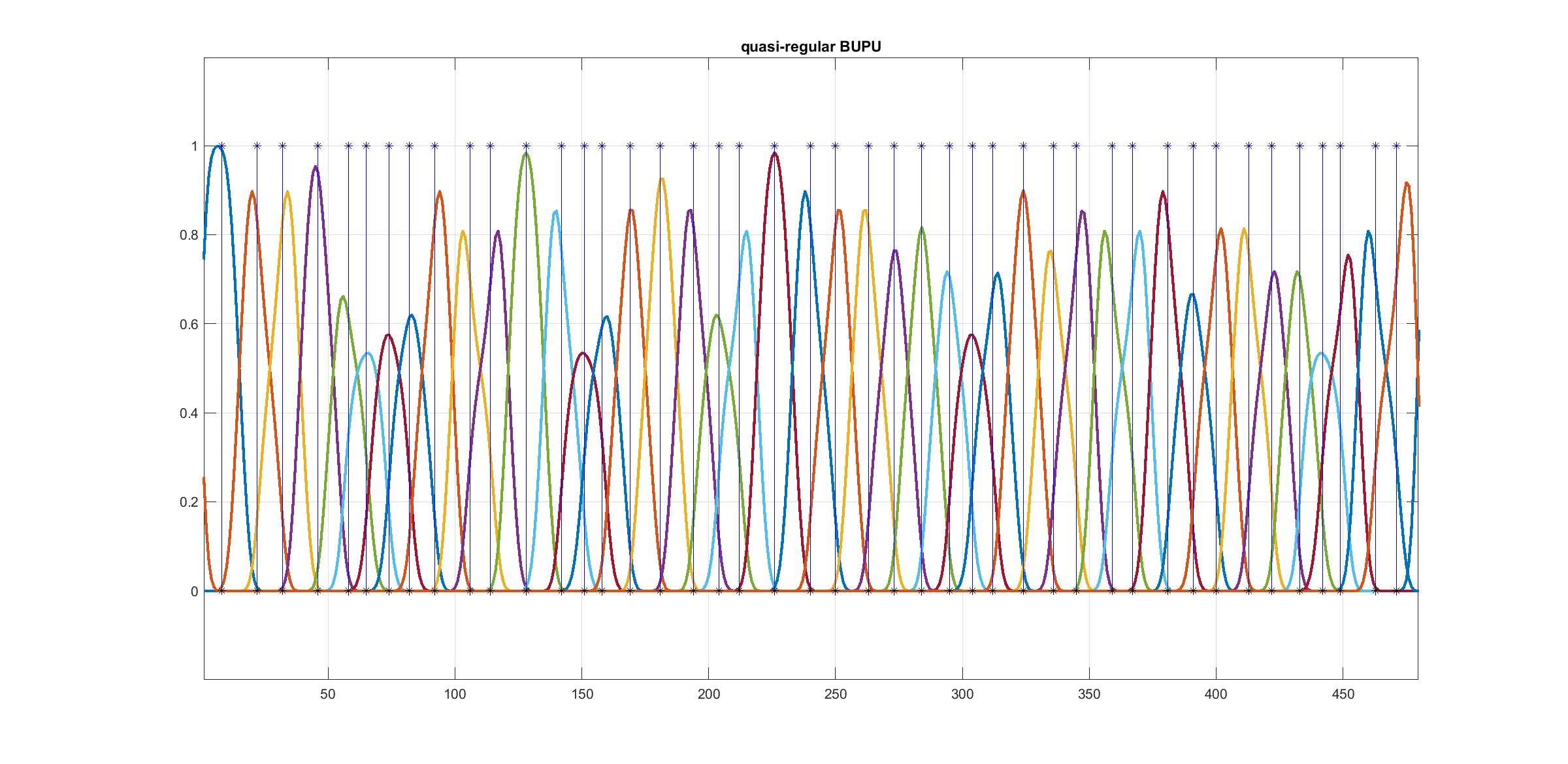}
} \end{center}
\caption{A typical BUPU, illustrating Definition \ref{BUPU-pos},
obtained by positioning shifted bump functions
at well-spread locations (marked with $*$) on the line,
followed by a division through the sum of those bump functions.}
\end{figure}

In contrast,  it is not
at all clear how to provide similar families of functions in a situation where there is a lack of fine lattices (and corresponding fundamental domains) and without having an appropriate group of automorphism on the underlying group (replacing dilations).

It is  our main goal in this section  to demonstrate that the existence of
such BUPUs (using a suitable version of the BUPU concept) can be guaranteed,
using relatively elementary arguments. Thus we will {\it not rely on the existence of a Haar measure} on such a group $\cG$, although that would make the proof a little bit shorter. 

The notion of a  \tblue{\it uniform partitions of unity}
appears in different papers, usually similar
in spirit and  mostly referring to a {\it uniform size of the constituents
of the partition of unity}. In order to compare the different possible
concepts  let us recall the corresponding definitions. The concept of choice for {\it this article} is that of BUPUs as introduced in \citeX{fe83} (i.e.\ Def. \ref{BUPU-gen0} below). It has been used
regularly since then (e.g.\ in \cite{fezi98}, section 3.2.2 and 
many other papers by the author).


The following situation will be the most simply and still most
useful for our purpose. It is a simplification of the concept
of BUPUs as introduced in \cite{fe83} (given below). Since it is
natural to formulate these results in the context of locally
compact groups $\cG$ we formulate the next definition by
writing the group operation in a multiplicative way.
\begin{definition} \label{BUPU-pos}
Given some neighborhood  $\UUG$ of the identity of a
locally compact group $\cG$, a  {\it non-negative U-BUPU}, a  so-called (left)
\tblue{{\bf b}ounded {\bf u}niform {\bf p}artition of {\bf u}nity of size $U$}
is a family $\Psi = (\psi_i)_{i \in I}$
of continuous, {\it non-negative  functions} on $\cG$
satisfying the following conditions (we write the group law multiplicatively
here):
\begin{enumerate}
\item For some family $\xiiI$ in $\cG$  one has:
$ \supp(\psi_i) \subseteq  x_i  U$  for all $i \in I$;
\item The family    $(x_i U)_{i \in I}$ satisfies the
{\it bounded overlap property} (BOP),     the number of intersecting neighbors  is uniformly bounded (with respect to $i \in I$):
$$\supiI \#\{j\suth x_i U \cap  x_j U \neq \emptyset \} \leq  B_0 < \infty;$$
\item  $\sum_{i \in I}  \psi_i(x) \equiv  1$ on $\cG$.
\end{enumerate}
\end{definition}

\begin{remark}
The continuity of the constituents $\psii$ of the BUPU require some
overlap of their supports which is illustrated by Fig.1.
 On the other hand we can apply bounded
measures (i.e.\ linear functionals on $\COGN$) only on continuous
functions with compact support, and not on the indicator functions
of compact sets. While one might think of a fine partition of the
group (e.g.\ translates of a {\it fundamental domain}) we do not
want to make use of this more measure-theoretic setting.
\end{remark}

\begin{remark}  Observe that the bounded overlap property implies
that the  sum in (3) is finite sum (with at most $B_0$ non-zero terms
for each $x \in \cG$). We call $B_0$ the ``overlap bound'' of the
family $(x_i U)$.

The non-negativity of the functions $\psii$ implies by (3)
that   $\supiI \infnorm{\psii} \leq 1$, i.e. the family
$\Psi$ is bounded in $\COGN$ (the space of continuous complex-valued
functions vanishing at infinity, endowed with the sup-norm).
\end{remark}


For the characterization of general Wiener amalgam spaces of the form
$\Wsp(\Bsp,\lqsp)$, for example,  (with a {\it local component} $\BspN$,
more general than just another $\Lpsp$-space, but something like
$\BspN = \FLispN$ or $\FLpsp$) it is important to assume boundedness of
the family $\Psi$ in some Banach algebra (with respect to
pointwise multiplication), contained in the multiplier algebra
of $\BspN$. We assume in that case that $\AspN \hkr \COGN$
(continuous embedding).
On the other hand, the non-negativity is not
required in this case. The subsequent definition of BUPUs
goes back to \citeX{fe83}.
\begin{definition} \label{BUPU-gen0}
Given $\UUG$, a family $\Psi = (\psii)_{i \in I}$  is a
 {BUPU},  a {\bf b}ounded {\bf u}niform {\bf p}artition of {\bf u}nity
(of size $U$) in the Banach algebra $\AspN$ if one has:
\begin{enumerate}
\item 
There exists a family $\xiiI$ in $\cG$ such that
$ \supp(\psi_i) \subseteq  x_i U$
for all $i \in I$;
\item The family $\Psi$ is bounded in $\AspN$,
i.e.\ $sup_{i \in I} \|   \psii \|_\Asp \leq C_\Psi < \infty$;
\item  There exists $B_0 > 0$ such that
$ \# \{ j \suth x_i  U \cap  x_j  U \neq  \emptyset \}  \leq  B_0;$
\item  $\sum_{i \in I}  \psi_i(x) \equiv  1$ on $\cG$.
\end{enumerate}
\end{definition}

The constant $C_\Psi = C(\Psi,\Asp)$ is  called the norm of the family $\Psi$ in $\AspN$, and $B_0$ is the {\it overlapping constant} of the family.
The family $X = \xiiI$ is called the {\it family of centers of the} BUPU $\Psifam$.

For the case of a metric group $\cG$ we can use balls of radius
$\delo$ as a basis of neighborhoods and thus it is natural to
write $|\Psi| \leq \delta$ if one has $\supp(\psii) \subseteq B_\delta(x_i)$
for $i \in I$. In this case we call $\Psi$ a $\delta-BUPU$, or
a BUPU of size $\delta$.

\begin{remark}
The usefulness of BUPUs with different specific properties arises in
various contexts. Let us mention  only a few of them here.

Sometimes it is enough to have {\it some} BUPUs, maybe bounded in a suitable
Banach algebra $\AspN$, e.g. for the construction of Wiener amalgam spaces,
such as $\WFLiliRdN$. In fact, for such spaces one can show that different
BUPUs define the same Wiener amalgam spaces with equivalent norms.
But BUPUs are not just helpful in defining new function spaces, they
play an important role in the alternative approach to convolution for
the measure algebra $\MGN$ as presented in \citeX{fe17}. The decomposition
of $\mu \in \MG$ as an absolutely convergent sum of measures with small
support allows to take a crucial step in the isometric isomorphism
between $\MGN$ and the Banach algebra (under composition) of bounded
linear operators on $\COGN$ which commute with translation, the so-called
TILS ({\it translation-invariant linear systems}, as they are called in engineering books).

For the work on {\it coorbit theory}  developed jointly with K.~Gr\"ochenig
as well as the closely related work on the {\it irregular sampling} it is
important to be able to have BUPUs which are centered at a given $\delta-$dense
subset of $\Rdst$ (or a LCA group), see \citeX{fegr89} or
\citeX{fegr89-2}, see also \citeX{fe91-1} and \citeX{fe92} and \citeX{fegr92-3}.

For the {\it current paper} the existence of {\it arbitrary fine} BUPUs over general LC groups will be crucial. Currently it is not clear whether one can find UPUs (in the sense of \citeX{lemu91})
 of size $U$ (meaning with $\supp(\varphi) \subseteq U$, for a given
neighborhood of the identity) in the case of general groups. Fortunately
 the concept of BUPUs  is more flexible, and it will be the first main
 result of this paper to demonstrate that one can derive the existence
 of arbitrary fine non-negative BUPUs over any given locally compact group $\cG$ using elementary considerations (reminding perhaps some readers of the
 construction of a Haar measure on $\cG$, see \citeX{we40}).

\end{remark}


For most applications so-called \tblue{\it regular BUPUs} will be sufficient (and in fact easier to handle), which are obtained as translates of a (smooth) function with compact support along some lattice  $\Lambda  \lhd \cG$. Especially over $\cG = \Rdst$ it would be natural to make use of smooth BUPUs with respect to some lattice of the form $\Lambda = {\bf A} \Zdst$, for some non-singular $d \times d$  matrix ${\bf A}$. Note that in the Euclidean
case (or for example also for stratified Lie groups) one can obtain
``arbitrary fine BUPUs'' by applying a simple dilation (or rather compression)
routine to a given BUPU. If one only needs {\it some BUPU} over $\Rdst$
it is quite natural to obtain BUPUs as translates of a single function:

\begin{definition} \label{regFL1BUPUd}
A family $\Psi = (\psi_\lambda)_{\lainLa} = (T_\lambda \psi_0)_{\lainLa} $ is
called a {\it regular (smooth) uniform partition of unity}  on $\Rdst$
in $\FLiRdN$
if it satisfies:
\begin{enumerate}\label{regBUPU1b}
\item  $\psi_0$ is compactly supported and $\widehat \psi_0 \in \LiRd$,
(resp.\  $\psi_0 \in \cD(\Rdst)$);
\item
$\sum_{\lainLa}\psi_\lambda(x)=\sum_{\lainLa}\psi_O(x-\lambda )\equiv 1$ on $\Rdst$.
\end{enumerate}
\end{definition}
Note that the finite overlap condition of support easily follows
from the properties of a lattice, and that furthermore the boundedness
of the family $\Psi$ is  an easy consequence of the isometric
translation invariance of the algebra $\AspN$ under consideration
(here $\FLiRdN$).


%

Historical note: BUPUs have been introduced (although not first used) by this name by the author in \cite{fe83} for the ``discrete'' characterization of Wiener amalgam spaces.

A slightly different approach has been taken in \citeX{lemu91}, based
on earlier work of \citeX{heno79}.

\begin{definition} \label{UPUdef1}
Let $\cG$ be a locally compact group.
A family $\Psi$ in $\CcG$  is called a UPU (a {\it uniform partition of unity})
if there exists some function $\varphi \in \CcG$\footnote{ i.e.
continuous and compactly supported, maybe satisfying some smoothness conditions.} such that, for a suitable family $(y_i)_{i \in I}$ in
$\cG$ one has
\begin{equation}\label{UPUprop}
\sumiI T_{y_i}\varphi(x)=\sumiI  \varphi(y^{-1}_i x)\equiv 1 \,
\, \mbox{on} \,\, \cG.
\end{equation}
\end{definition}
Although formally there is no BOP-property required in this case
 it is shown in \citeX{lemu91} that the family of shift parameters
 $(y_i)_{i \in I}$ is relatively separated, or equivalently, that such an UPU is in fact a BUPU of size $\supp(\varphi)$.

We can give the following characterization of {\it relatively separated
families} $\xiiI$ as they appear in the above definitions.
 For details see e.g.\ Thm. 22 in \cite{fe92}.

\begin{lemma}\label{relsep02}
For a discrete family $\xiiI$ in  $\cG$ the following properties
are equivalent:
\begin{enumerate}
\item The family is {\it relatively separated}, i.e.\ a finite union
of separated sets, i.e.\ of subfamilies $(x_j)_{j \in J}$ with the
property that $x_j V \cap x_l V = \emptyset$ for $j \neq l \in J$,
for some open set $V$ in $\cG$;
\item For any relatively compact set $W$ the family $(x_i W)_{i \in I}$
has the uniformly controlled neighbors property;
\item For any compact set $Q \subset \cG$  the number of points in $zQ$  is
controlled as follows:
$$  \sup_{z \in \cG} \#\{ i \suth x_i \in zQ \} = B(Q) < \infty. $$
\end{enumerate}
\end{lemma}
\begin{remark}
The last property can be equivalently described
as the property that the (irregular) Dirac
comb $\shahf_{\, X} := \sumiI \delta_{x_i}$
belongs to the Wiener amalgam space $\WMlinf$, which is the
dual of the Wiener algebra $\WCOliGN$. We will not pursue
this connection any further as it might confuse the readers
not familiar with the theory of Wiener amalgam spaces.
\end{remark}

Let us next remind  that the main result of
\citeX{lemu91} describes (making use of the structure theory
of locally compact groups) the existence of UPUs for
arbitrary LC groups $\cG$. However, it is not claimed
that one can find arbitrary fine UPUs in that paper. Still,
for further reference let us formulate their main result
as follows:
\begin{proposition}[Leptin/M\"uller]  \label{lemu91A}
For any locally compact group $\cG$ there exist UPUs.
\end{proposition}
\begin{remark} \label{UPUtoBUPU}
Using a simple compactness argument one can even rewrite the function
$\varphi$ as a finite sum $\varphi = \sum_{k=1}^K \varphi_k$
of functions with arbitrary small support
and thus derive the existence of BUPUs by translating each of them
using the same family of shift-parameters $(y_j)$. However, the
disadvantage (from our perspective) of this approach is the fact that it is heavily based on structure theory. 
\end{remark}

\begin{remark} \label{tiling}
Note that of course one even can have the situation, where
the indicator function of a relatively compact set covers
the group by translates along a discrete family $(y_j)_{j \in J}$,
without having any group structure, i.e.\ not using a lattice
(discrete subgroup) as parameter set of the shift operators.
Such a situation is known from wavelet theory, where
one obtains such coverings on the ``$ax+b$''-group. Although
the translation parameters (taken from $\Zst$) and the dilation
parameters of the form $2^k, k \in \Zst$ form discrete subgroups
of Abelian subgroups the combined ``geometric lattice'' is {\it
not a discrete subgroup} of the affine group.
\end{remark}

\subsection{Historical Notes}  \label{sechist}

There are several situations where BUPUs have
played an important role in the past. First of all the paper
introducing the general {\it Wiener amalgams} (originally called
{\it Wiener-type spaces}) in \citeX{fe83}.

Of course various forms of smooth BUPUs, such as B-spline systems
have been used already early on e.g.\ in the theory of numerical
integration. In fact, any BUPU allows to define a so-called
\tblue{\it quasi-interpolation operator} of the form
$\SpPsi f(x)  := \sum_{i \in I} f(x_i)\psii$.
Sometimes (e.g. for the BUPU obtained by B-splines of order one,
which are triangular functions) these operators interpolate
the function $f$ at the node points, but in most other
cases they just approximate a given smooth function.
Integration formulas thus allow to calculate
the integral of $\SpPsi$ in a closed form, based on the
knowledge of the sampling values $(f(x_i))$ only.

BUPUs over LC groups play a prominent role in the development
of {\it coorbit theory} put in place by the author together with
K.~Gr\"ochenig (see \citeX{gr01}).

At the heart of coorbit theory are reconstruction methods
which allow to reconstruct an abstract wavelet transform $V_gf$
defined over a locally compact group  $\cG$ (such as the Heisenberg group,
the ``$ax+b$''-group, or the shearlet group, to mention concrete
examples) from samples, just taken over a sufficiently dense,
discrete family $(x_i)_{i \in I}$ in $\cG$.  The first step
here is to establish a quasi-interpolation for $V_gf$, using the
given sampling values only. But then one observes that the
resulting function may not belong to the range of the transform
$f \mapsto V_gf$, and thus one has to project back to the
range, which can be realized by a convolution with $V_gg$.
The details are found in \citeX{fegr89} and related papers.
Of course one has to estimate the guaranteed approximation
quality for this kind of approximation, in order to have a
basis for an iterative method of reconstruction (at a geometric rate).

The intuitive similarity of the properties of those wavelet
transforms $V_gf$ with band-limited functions of two variables
(also the existence of a reproducing convolution relation in
both cases) where then inspiring the authors to deal with
the ``irregular sampling problem'', i.e.\ the problem of reconstructing
a band-limited function from irregular samples. Recall that
the regular case, i.e.\ reconstruction of a band-limited
function in $\LtRd$ (with compact support $\supp(\hatf) \subseteq  B_R(0)$)
from samples along some lattice $\Lambda$ can be guaranteed if the lattice
$\Lambda$ is fine enough, essentially making use of Poisson's formula.
In the irregular case the first generation of iterative algorithms was
based on the use of BUPUs which are fine
enough and are centered at the given sampling points (see \citeX{fe92}, \citeX{fegr92-3}).


\section{Arbitrary fine BUPUs over LC Groups}    \label{secCAPBUPU}

In this section we establish our first main result, in the context
of general locally compact groups. Since this includes many
non-commutative groups we choose the usual multiplicative notation
for the group law.

\begin{definition} \label{capSdef1}
For any fixed and relatively compact subset $S \subset G$ the
mapping $M \mapsto \capS(M)$ is defined on the collection of (relatively)
compact sets $M$ by the following rule
\begin{equation}\label{capsDef}
\tblue{ \capS(M):= \min \{ \#F \suth M \subseteq \bigcup_{i \in F} \, x_i S\} }
\end{equation}
where the minimum is taken over all finite subsets of possible translation
parameters.
\end{definition}
Note that by the fact that the interior of $S$ is non-empty and $M$ is supposed
to be compact the minimum is taken over a non-empty subset of $\Nst$.

\begin{remark}\tblue{ The term ``capacity'' originates from a similar construction, where
one measures the size of an indicator function by minimizing over all the
(typically non-negative) functions in a given function spaces, typically
a Sobolev space $\HsRd$, which dominate the indicator function
${\bf 1}_M$ of the set $M$.}

Such an interpretation is in fact possible also here:
Given a set $S$,  $\capS(M)$ can be interpreted as the
infimum over all norms in $\WLinfli(\cG)$ of functions, dominating the
indicator function  ${\bf 1}_M$. We leave it to the interested reader
to check the details.
\end{remark}

\begin{lemma} \label{capSrules1}
For any fixed and relatively compact subset $S \subset G$ the
mapping $M \mapsto \capS(M)$, defined on the collection of
compact sets $M$ has the following properties:
\begin{enumerate}
\item $\capS(S) = 1$ 
\item $\capS(zM) = \capS(M), \,\,\forall  z \in \cG$
\item The mapping $M \mapsto \capS(M)$ is subadditive in the sense
that we have for finite unions of compact sets:
\begin{equation} \nonumber
 {\capS \left(\bigcup_{k=1}^K M_k \right) \leq  \sum_{k=1}^K \capS(M_k)}
\end{equation}
\item Given any  finite collection of compact sets $M_k, k = 1,...,K$
which are $S$-{\it separated}, i.e.\ satisfying the condition that $M_k S \cap  M_j S = \emptyset$
for $k \neq j$ one has
$$  {\capS\left(\bigcup_{k=1}^K M_k \right) =  \sum_{k=1}^K \capS(M_k)}. $$
\end{enumerate}
\end{lemma}
\begin{proof}
Claim (1) is obvious, and the  translation invariance (2) follows from
$$ M \subseteq \bigcup_{i \in F} \, x_i S \quad \Leftrightarrow
\quad zM \subseteq \bigcup_{i \in F} \, (z x_i) S, \quad \forall z \in \cG.$$
Thus any covering of $M$ has a corresponding
covering  of equal cardinality for $zM$.

The subadditivity property (3) is easy to check, since the combination
of all the translates needed to cover all the sets $M_i, i \in F$,
is obviously constituting a covering of their union.

Finally we check for the additivity property (4). 
Given a minimal covering of $\bigcup_{k=1}^K M_k$,
using a set of translates of the form $y_i S, i = 1,...  L$, we argue
that each of the translates will be relevant for exactly
one of the constituting sets $M_k$, $ 1 \leq k \leq K$, since by the
minimality we can  limit our consideration to translates
of the form $zS$ which intersect at least one of the sets $M_k$.

Using an indirect argument we assume that
$ zS $ has non-trivial intersection with, say, $x_1 S$ and $x_2 S$.
Then we have $ zs = x_1 s_1 = x_2 s_2$ for some elements $s,s_1,s_2 \in S$.
But then we have  $z \in M_1 S \cap M_2 S = \emptyset$, in
contradiction to the assumption.
Thus for every index $k$ the collection of sets $s_i S$ with $i \in I_k$
given by
$$  I_k := \{ i \in F \suth  M_k \cap s_i S \neq \emptyset \} $$
describes a   covering of the set $M_k$.

It is a minimal covering, because if there
was another covering of the set $M_k$ with fewer terms, it could
be used to obtain an even better covering of their union (by just
leaving the other contributions fixed), in contradiction to the
assumed minimality of the covering and  property (3)
in Lemma \ref{capSrules1}.
\end{proof}

\begin{remark}
Note that up to this point we have only used a few topological
properties of locally compact groups $G$. The use of the simple
expression of a {\it capacity} (which should be seen as a simplified
or coarse form of a measure) will allow us to derive the existence
of arbitrary fine BUPUs on any locally compact group $\cG$.

Note that similar expressions appear in the construction of the
Haar measure on a given locally compact group. We leave it to the
reader to check this similarity. For us it is only important to
mention that the use of this ``coarse form of a measure'' precedes
the construction of a Haar measure and thus allows to derive the
validity of the integrated group action (as described in Section 3.2
of \citeX{fo95})  without any measure theory,
as it  {does not make use of Lebesgue integration theory nor the
existence of the Haar measure}. 
\end{remark}

\begin{theorem} \label{BUPULCG1}
Let $\cG$ be any locally compact group and $\UUG$ be any neighborhood
of the neutral element $e \in \cG$. Then there exist (plenty of) BUPUs
$\Psifam$ of non-negative functions of size $U$, meaning that 
\begin{equation} \label{suppcond1}
 \supp(\psi_i) \subseteq  x_i U, \quad \forall i \in I,
\end{equation}
for a suitable discrete (in fact uniformly separated)
family $X = (x_i)_{i \in I}$ in $\cG$, and
\begin{equation}\label{partunit01}
   \sum_{i \in I} \psi_i(x) \equiv 1.
\end{equation}
\end{theorem}
\begin{proof}
Given $U$ we choose some compact  neighborhood $V \in \UnbG$ such that $V^3 \subseteq U$, and an even smaller neighborhood $S \in \UnbG$ with $ S^2 \subseteq V$. Without loss of generality we will assume that all these neighborhoods are symmetric (with respect to the group action), i.e.\ that $ z \in U$ if and only if $z \inv \in U$ (and the same of the other neighborhoods).

We then select a maximal family $(x_i)_{i \in I}$ with respect to the
property that $ \{ x_i V \suth i \in I \}$ forms a {\it pavement} in $\cG$, i.e.\  such that the sets $x_i V$ do not intersect in a non-trivial way, but that there is no $z \in \cG$ such that $z V$ could be added to the family without destroying this property. Consequently, any translate $zV$ intersects at least one of the sets $x_iV$, or  $z V \cap x_iV \neq \emptyset$ for some index $i \in I$. By the symmetry assumption this implies that
 the family $(x_i V^2)$ covers the group $\cG$.

Due to the regularity of a locally compact group there
exists $ \varphi \in \CcG$ with $\varphi(y) = 1$ on $V^2$,
$\supp(\varphi) \subseteq V^3$ and $\infnorm{\varphi} = 1$.
Thus  (by setting $L_x \varphi(y) = \varphi(x \inv y)$)
 the sum
$$ \Phi(x):=  \sum_{i \in I}  L_{x_i} \varphi(x)
=  \sum_{i \in I} \varphi({x_i}\inv x)  $$
is well defined and satisfies $\Phi(x) \geq 1$ for all $x \in \cG$.
In order to show that the sum is finite (in a uniform sense) for each $x \in \cG$, let us fix $i \in I$ and consider
$I_i := \{ j \suth  \psi_i \psi_j \neq 0 \}.$
%
Since $ I_i \subseteq \{ j \suth x_i V^3 \cap x_j V^3 \neq \emptyset\} $
we have to count the indices $j \in I$ with
 $ x_j \in x_i V^6$, or
$$ x_j S \subseteq x_i V^6 S \subset  x_i V^7. $$
Since the family of sets $ x_j S, i \in I$ is an $S$-separated
family of translates of $S$, thanks to the assumption $S^2 \subseteq V$
and the pavement conditions stated at the beginning of the proof
we can apply property (4) of  Lemma \ref{capSrules1} in order to
finish our proof.

Hence for any fixed $i \in I$ the number of possible indices such that
$\psi_i \cdot \psi_j \neq 0$ is at most   $\capS(V^7)$ (because
we have $\capS(x_j S) = \capS(S) = 1$, using properties (1) and (2)
from Lemma \ref{capSrules1} above).

Overall we have established that the sum defining $\Phi(x)$ is pointwise a
finite sum and the resulting function $\Phi$ satisfies
\begin{equation} \label{Phiprops1}
\tblue{ 1 \leq \Phi(x) \leq \capS(V^7), \quad x \in \cG.}
\end{equation}
Consequently we observe that
the family defined by
$$\tred{\psi_i(x) := \L_{x_i}\varphi(x)/\Phi(x)}, \,\,  i \in I,$$
defines a partition of unity of size $U$, since $V^3 \subseteq U$ and
$$\tred{\supp(\psi_i)=\supp(L_{x_i}\varphi) = x_i \supp(\varphi) \subset x_i V^3 \subset x_i U}.$$
\end{proof}

\section{Towards Integrated Group Representations}
\label{OpBUPUs}

To some extent the usefulness of BUPUs is based on the fact that
they allow to define natural operators. Any non-negative BUPU $\Psifam$  
induces two operators namely the {\it spline quasi-interpolation operator} $\SpPsi$ on $\COGN$, given by
\begin{equation}\label{spindef02b}
\SpPsif  := \sum_{i \in I} f(x_i)\psii,
\end{equation}
and its adjoint operator, the so-called {\it discretization operator}
$\DPsi$ on $\MGN = \COPGN$, which takes the form
\begin{equation}\label{dpsidef02b}
 \DPsimu = \sumiI \mu(\psii) \delta_{x_i}.
\end{equation}
Since any $\SpPsi$ is obviously a nonexpansive operator on $\COGN$
it is also clear that its adjoint is nonexpansive on $\MGN$ as well.

Let us first recall a few facts concerning the discretized measures
for the case of $\cG = \Rdst$.
%
%
In \citeX{fe17} the following facts have been derived:
\begin{proposition} \label{Dpsiconv03}
Given $\mu \in \MRd$ the net\footnote{The reader is definitely familiar
with such a concept, recalling the concept of convergence of Riemannian sums,
which approach the limit $\int_a^b f(x)dx$, given $f \in \Csp([a,b])$.}
$(\DPsimu)_{\sPsitoz}$ is $\wst$-convergent:
\begin{equation}\label{wsconvDPsimu2}
\DPsimu(f)  \to \mu(f), \quad \forall f \in \COsp, |\Psi| \to 0.
\end{equation}
In fact, we have for any BUPU  $\Psi$:
\begin{equation}\label{normcontr2}
\Mbnorm{\DPsimu} \leq \Mbnorm{\mu}, \quad \mu \in \MRd.
\end{equation}
Moreover, the family  $(\DPsimu)_{|\Psi| \leq 1}$ is uniformly
tight in $\MRdN$\footnote{A bounded set $S \subset \MRd$ is called
tight if for every $\epso$ there exists $p \in \CcG$ such that
$\Mbnorm{p\mu-\mu} \leq \veps, \forall \mu \in S$.}. 
\end{proposition}

We do not go into a discussion of tightness combined with $\wst$-convergence,
but recall that we have  established  strong operator norm convergence
for the corresponding convolution operators (by \cite{fe17}), given pointwise by $ \mu \ast f (x) = \mu(T_x \fchk), \,\, x \in \Rdst$:
\begin{equation} \label{convconvf}  \tred{
\limPsitoz \infnorm{ \DPsimu \ast f - \mu \ast f} = 0, \quad f \in \CORd}.
\end{equation}

Our next goal is to verify that a corresponding behaviour persist
to be valid for  general (isometric) group representations on Banach
spaces. In a sense, this shows that the
Banach algebra $\MGN$, with the composition rule being internal convolution),
provides a universal algebra which can be embedded into the Banach algebra
of all operators on a variety of Banach spaces. Note that in addition
to the crucial estimate (controlling the operator norm of the convolution
operator $f \mapsto \mu \ast f$ by $\Mbnorm{\mu}$) we have to ensure
the validity of the associative law, i.e.\ that we have
for $\mu_1,\mu_2 \in \MG$:
\begin{equation}\label{assocconv02}
(\mu_1 \star \mu_2) \ast f  = \mu_1 \ast (\mu_2 \ast f), \quad f \in \Bsp.
\end{equation}
It is  non-trivial and often forgotten to mention, but it is obvious
for Dirac measures, hence for discrete measures, and thus can be
obtained by taking limits.

Since our goal is mostly the application for LCA groups we
formulate the next definition for the Abelian setting, thus
making use of additive notation for the group law.
\begin{definition} \label{BanRepr1}
A mapping $\rho: \cG  \to \mathcal{L}(\Bsp)$,
the bounded, linear operators on a
Banach spaces $\BspN$,  is called an \tblue{\it isometric  representation}
of a group  $\cG$ on the Banach space $\BspN$ if the mapping
$\rho$ is a group homomorphism, i.e. satisfies
$$ \rho(x+y) = \rho(x) \circ \rho(y), \quad x,y \in \cG,$$
and if each of the operators are isometric on $\BspN$, i.e. if
one has
\begin{equation}\label{isomrepr2}
\tred{  \Bnorm{ \rho(x) f} = \Bnorm{f}, \quad f \in \Bsp, x \in \cG. }
\end{equation}

Moreover, if  the mapping $ x \mapsto \rho(x)f$ is continuous from
$\cG $ to $\BspN$, i.e.\
\begin{equation}\label{strongcont2}
\tblue{ \lim_{x \to 0} \Bnorm{ \rho(x)f - f} = 0, \quad f \in \Bsp.}
\end{equation}
we say that the representation $\rho$ is \tblue{\it strongly continuous}.
\end{definition}

An important family of examples arises from the so-called
{\it regular representation} of $\cG$, i.e.\ the action of the
group by (left or right) translation on functions or distributions
over $\cG$, i.e. $\rho(x) = T_x$ the integrated
action corresponds to the {\it usual convolution} (see \citeX{fo95},p.73).
In this case the notation of
{\it homogeneous Banach spaces} is used, suggesting to call Banach
spaces endowed with an isometric, strongly continuous group
representation of a LC group $\cG$ an 
{\it abstract homogeneous Banach space} (cf. \citeX{sh71}, Chap. 9).

The main result of this paper is the observation that we can establish
the fact  that every strongly continuous, isometric representation
of $\Rdst$ on a Banach space $\BspN$ gives rise to an extended
representation (the so-called {\it integrated group representation})
of the Banach convolution algebra $\MRdN$. In fact, this extension
is unique among all those who respect tight, $\wst$-convergence of
nets (or just sequences), with the understanding that
$\rho(x)$ is of course identified with $\rho(\delx)$\footnote{We avoid
the use of a different symbol for the integrated representation.}.

\begin{remark}
Usually, in the standard literature on the subject, the integrated
group representation describes the action of $f \in \LiG$ on $f \in \Bsp$,
and is thus not immediately visible as a natural extension of the
group representation. Aside from technical arguments (and there are
many such considerations, involving abstract measure theory and a lot
of functional analysis) the focus on $\LiG$ appears to come from a
similar situation, where the group representation of a discrete group $\cG$ can be extended naturally to $\lisp(\cG)$, which has the ``unit vectors'' $\delta_x, x \in \cG$ as a natural (unconditional) basis. In other words, in this case any $f \in \lisp(\cG)$ can be written (uniquely) as
\begin{equation} \label{liG04}
 f = \sum_{x \in \cG} f(x) \delta_x \quad \mbox{with} \quad
   \linorm{f} := \sum_{x \in \cG} |f(x)| < \infty.
\end{equation}
But for a discrete group we have of course $\lisp(\cG) = \Msp(\cG)$ and the
finite, discrete measures are dense in
$\MGN$ (see \cite{si15-2}, Example 6.1.7). In contrast, for
non-discrete groups the subspace  $\MdG$  of discrete
measures (of the form $\mu = \sumkinf c_k \delta_{x_k}$ with
$ \sumkinf |c_k| < \infty$) forms a proper closed subalgebra of $\MGN$.
However, fortunately  $\MdG$ is $\wstd$dense in $\MG$ and the constructive way
of proving this fact (described in \citeX{fe17}) serves is the basis for
the results presented in this paper.
\end{remark}

\begin{remark}
Using the terminology of Banach modules we can state: Any strongly
continuous, isometric representation of $\cG$ on $\BspN$ turns
$\BspN$ into a Banach module over the (commutative, unital) Banach
convolution algebra $\MGN$ (we use the symbol $\star$ for internal
convolution).

Later on (see Section \ref{secHBS}) we will see that the
restriction of the module action to $\LiG$
makes $\BspN$  an {\it essential Banach (convolution) module over } $\LiGN$.
\end{remark}

\vspace{2mm}


Next we  will show that the convolution action of bounded discrete measures on a  {\it homogenous Banach space} can be extended to all of the measures in order to generate an action of $\MGN$ on such a Banach space $\BspN$.

\begin{theorem} \label{conv-meas-homBR0}
Any abstract homogeneous Banach space $\BspN$ with respect to a given,
strongly continuous and isometric representation $\rho$ of a locally compact
group $\cG$ is also a Banach module over the Banach algebra $\MGN$ (with
respect to convolution). This claim includes the validity of following
associativity law:
\begin{equation} \label{convassoclaw}
\rho(\mu_1 \star \mu_2) = \rho(\mu_1) \circ \rho(\mu_2), \quad
\mu_1,\mu_2 \in \MbG. 
\end{equation} 
The mapping $(\mu,f) \mapsto \mu \rhobul f = \rho(\mu)f$ 
is the natural extension of
the action of discrete measure given by $\delta_x \rhobul f = \rho(x)f$,
and satisfies the norm estimate
\begin{equation}\label{meas-onB1}
\|\mu \rhobul f\|_\Bsp\leq\|\mu\|_\Msp \fBN,\quad  \mu \in \MG,\,f \in \Bsp.
\end{equation}
\end{theorem}

\begin{proof}
We start from the expected action of  Dirac measures via
\begin{equation}\label{rhobuldelta}
\delta_x \rhobul f = : \rho(x)f, \quad f \in \Bsp.
\end{equation}
Since discrete measures are absolutely convergent sums of Dirac measures it
is then clear that we have for a discrete measure
$\mu = \sumkinf c_k \delta_{x_k}$, with
$\sumkinf |c_k| = \Mbnorm{\mu} < \infty$:
\begin{equation} \label{rhomudisc1}
\mu \rhobul  f  = \sumniinf c_k  \rho(x_k) f,
\end{equation}
the sum being absolutely convergent for each
$f$ and $\mu \in \MdG$, since we have
\begin{equation} \label{rhomudisc2}
\Bnorm{\mu \rhobul  f}  \leq  \sumniinf |c_k|  \Bnorm{\rho(x_k) f}
\leq  \Bnorm{f} \sumniinf |c_k| = \Mbnorm{\mu} \Bnorm{f}.
\end{equation}

Observe also that the assumptions concerning $\rho$ imply that this action
of $\MdG$ is not just an individual action (given for each $\mu \in \MdG$)
but it defines in fact a representation of the Banach convolution algebra
$\MdGN$, since we have  
\begin{equation}\label{rhobulass1}
(\mu_1 \star \mu_2) \rhobul f =  \mu_1 \rhobul (\mu_2 \rhobul f),
\quad    \mu_1,\mu_2 \in \MdG, \, f \in \Bsp
\end{equation}
as a consequence of the validity of
\begin{equation}\label{assocDirac1}
(\delta_x \star \delta_y) \rhobul f = \delta_{x+y} \rhobul f
= \rho(x+y) f = \rho(x) (\rho(y) f) = \delta_x \rhobul (\delta_y \rhobul f).
\end{equation}
Consequently, for given $\mu \in \MbG$ and  $\finB$ we set
\begin{equation}
 \DPsimu \rhobul f  =  \sum_{i \in I}  \mu(\psi_i) \rho(x_i) f.
\end{equation}
Based on (\ref{rhomudisc2}) and (\ref{normcontr2})  we have for any $\Psi$:
\begin{equation}\label{discrmeasconv1}
\| \DPsimu \rhobul f \|_\Bsp  \leq
\fBN \sum_{i \in I} | \mu(\psi_i)|  =  \fBN  \Mbnorm{\DPsimu}
  \leq\Mbnorm{\mu}  \fBN.
\end{equation}

We will show next that it is convergent, as $ |\Psi|  \to 0$
or $\diam(\Psi) \to 0$.
The motivation for this approach becomes plausible once one understands  $\DPsimu$ on $f$ as a Riemann-type sum for the Banach-space valued integral of $x \to \rho(x)f $, usually written as $  \int_\cG  \rho\ofp{x} f\ofp{x} d \mu(x) $.

Given two families $\Psifam$ and $\Phifam$, with their centers
$(x_i)_{i \in I}$ and $(y_j)_{j \in J}$ respectively.
We define their {\it joint refinement}
$ \Psi-\Phi$ as the family   $(\psi_i \phi_j)_{(i,j) \in I\diamond J}$.
It is natural to take   $I \diamond J$, the family of all index pairs such
that $\psi_i \cdot \phi_j \neq 0$ (because all the other products are trivial and should be neglected) as the new index set.
In fact, if both $\Psi$ and $\Phi$ are sufficiently ``fine'' BUPUs one has:
\footnote{Using  that  $\psi_i = \sum_{j \in j}  \psi_i \phi_j$, hence $\sum_{(i,j) \in I \diamond J}  \psi_i \phi_j  \equiv  1$ and
$\sum_{(i,j) \in I \diamond J}   \|(\psi_i \phi_j) \mu\|_\Msp = \| \mu\|_\Msp.$}

\begin{equation}\label{Riemann-estim1}
\|\DPsimu \rhobul f -\DPhimu \rhobul f\|_\Bsp = \sum_{(i,j) \in I \diamond J}
\|\rho(x_i) f - \rho(y_j) f\|_\Bsp  |\mu(\psi_i \phi_j)| \leq
\end{equation}
$$
\sup_{(i,j) \in I \diamond J}  \| \rho(x_i)[f - \rho(y_j - x_i) f] \|_\Bsp
\sum_{(i,j) \in I \diamond J}    \|(\psi_i \phi_j) \mu\|_\Msp \leq \vareps  \| \mu\|_\Msp,
$$
if only $  \Psi$ resp.\ $\Phi$ are fine enough.
Due to the completeness of $\BspN$ one finds that
there is a uniquely determined limit, which we will
call $\mu \rhobul f$. It is then obvious that
\begin{equation}\label{muastf-estim1}
\| \mu \rhobul f\|_\Bsp = \lim_{|\Psi| \to 0} \| \DPsimu \rhobul f \|_\Bsp  \leq \limsup_{\sPsitoz} \|\DPsimu\|_\Msp \|f\|_\Bsp = \muMN  \fBN.
\end{equation}
Of course it remains to show that the
action defined in this way is associative, i.e.\ that
\begin{equation}\label{assocat}
(\mu_1 \star \mu_2)\rhobul f = \mu_1 \rhobul (\mu_2 \rhobul f),
\quad \forall \mu_1,\mu_2 \in \MG, f \in \Bsp,
\end{equation}
but this follows from the associativity for the
discrete measures $\DPsimu$ and $\DPhimu$.
\footnote{Note that H. S. Shapiro
(cf.\ \cite{sh71}) is making this associativity an {\it extra axiom},
apparently because he could not prove it directly for technical
reasons, based on the way how he defines the action of a bounded measures
on an ``abstract homogeneous Banach space''. H.C.~Wang exhibits
in \cite{wa77} an example of what he calls a
{\it semi-homogeneous Banach space}
(without strong continuity of the action of $\cG$ on $\BspN$,
which does not allow the extension to all of the bounded measures.
Indeed, it is a Banach space of measurable and bounded functions on $\Rst$ which is non-trivial, but which does not contain
any non-zero {\it continuous} function! The example was suggested
to him in a correspondence by the author of this note. }
\end{proof}

\begin{remark}
In the derivation above we have used the isometric property and the fact that $\rho(x_1 x_2) = \rho(x_1) \circ \rho(x_2)$.
It would have been no problem if this identity was only true ``up to some constant of absolute value one'', i.e.\ if
one has a {\it projective representation} of $\cG$ only,
such as the mapping
$\lato  \mapsto  \rho(\lambda) = M_\omega T_t $
from $\TFRd$ into the unitary operators on the Hilbert space
$\LtRdN$, which is one of the key players in
{\it time-frequency analysis}. This direction will also be
explored further in subsequent notes.
\end{remark}

\begin{remark}
Another possible and powerful extension of the above result
will involve cases where the group action is not isometric anymore,
but still bounded by some weight function, i.e. the case where
each $\rho(x)$ is a bounded operator and one has control on the
operator norms of these operators on $\BspN$. In this case one
has to replace the algebra $\MGN$ by weighted versions, and
$\LiGN$ by Beurling algebras (see \citeX{re68}). Also this
direction will be pursued elsewhere in more details. It is a
crucial starting point for the analysis of TMIBs, i.e.
translation and modulation invariant function spaces (see
e.g. \citeX{dipivi15-1},\citeX{fegu21}).
\end{remark}

\section{The Wiener Algebra $\WCOlisp$}   \label{secWG}

The purpose of this section is to demonstrate that the concept
of homogeneous Banach spaces over LCA groups, originally introduced in the
book of Y.~Katznelson \cite{ka76} (see Remark \ref{katzrem}  below),
 can be introduced {\it without} making use of the Haar integral.
For this purpose we will make use of Wiener's algebra (as described
in  \cite{fe77-3}),
which is found already in Reiter's book (\citeX{re68}
and \citeX{rest00}) for $\cG = \Rdst$,
as a prototypical example of a Segal algebra.
It was the model case for many characterizations of minimal spaces
(pointwise $\CORd$-module in this case), see \citeX{fe77-3}, and
the subsequent papers \citeX{fe81} and  \citeX{fe87-1}.

Obviously BUPUs play an important role in the description of
{\it Wiener amalgam} spaces (such as Wiener's algebra, which
is of the form $\WCOliRd$, or the Segal algebra $\SORd = \WFLiliRd$).
The justification for characterizing Wiener amalgam spaces via BUPUs
comes from the main results of \citeX{fe83}.
Leaving out details let us summarize
a few properties of Wiener's algebra on a general LCA group $\cG$:
\begin{definition} \label{WienerAlgdef2}
\tblue{ $$ \Wsp(\cG) := \WCOlisp(\cG) := \{ f \in \COG \suth
\normta f \Wsp :=  \sumiI \infnorm {f \psii} < \infty \}.$$}
\end{definition}
We have the following general facts, which are easily proved
without making use of the existence of a Haar measure on $\cG$:
\begin{proposition} \label{Wienerprops1}
\begin{enumerate}
\item 
$\WspN$ is a Banach space, for any BUPU $\Psi$,
and continuously embedded into $\COGN$.
\item $\WspN$ is a Banach ideal in $\COGN$, i.e.\
pointwise products are in $\Wsp$, in particular it is a
Banach algebra under pointwise multiplication;
\item The space does not depend on the particular choice of $\Psi$,
i.e.\ different BUPUs define the same space and equivalent norms;
\item The decomposition of $f \in \Wsp$ as $f = \sumiI f \psii$
is not only valid absolutely in $\COGN$, but even in $\NSPB \Wsp$.
Hence $\CcG$ is dense in $\NSPB \Wsp$ and $\Wsp$ is dense in $\COGN$;
\item For {\it any open, relatively compact neighborhood $Q$ of the identity}
we have the following {\it atomic characterization} of $\Wsp$, via absolutely
convergent series:
$$ {\Wsp_{\nnth  \negthinspace at} := \{ f \in \COG \suth  f = \sumkinf f_k, \,\, \mbox{with} \, \sumkinf \infnorm{f_k} < \infty, \exists x_k \in \cG:
\supp(f_k)\subseteq x_k+Q \}. }$$
\item The corresponding (equivalent) {\tt inf}-norm (infimum over all admissible  sums) is isometrically translation invariant,
    with continuous translation, i.e.
\begin{equation}\label{Wsptransl1}
\normta {T_x f} \Wsp = \normta{f} \Wsp, \, x \in \cG, \quad \mbox{and} \quad 
\lim_{x \to e} \normta { T_x f - f} \Wsp = 0,
\quad \forall f \in \Wsp.
\end{equation} 
\end{enumerate}
\end{proposition}

\begin{remark}
As a matter of fact the functions in $\Wsp(\Rdst)$ are (even absolutely
Riemann) integrable and thus $\Wsp$ is a dense subspace of $\LiRdN$.
Combined with property (\ref{Wsptransl1}) this implies that
$\WCOlisp(\Rdst)$  is in fact a {\it Segal algebra} on $\Rdst$
(see \citeX{re68,rest00}). Similar comments apply for general LCA
groups  based on Prop. \ref{Wienerprops1}, once the existence of a
Haar measure is established (in order to characterize $\LiGN$ as
a closed ideal of $\MGN$, see Section \ref{secHBS}).
\end{remark}

Since $\WspN \hkr \COGN$ as a dense subspace, the dual space
$\Wsp^*$ (which can be characterized as the subspace $\WMlinf(\cG)$
of all Radon measures) is known in the literature as the space of {\it translation-bounded measures}.

First we give a characterization of $\Wsp^*$ as a subspace of all
tempered distributions (for the case $\cG = \Rdst$). Note that
in this case $\ScRd$ is a dense subspace of $\Wsp$.
\begin{lemma}\label{WBdual01}
A tempered distribution $\siScP$ extends to a bounded linear functional
on $\WspN$ if and only if one has the following estimate:

\tblue{Fixing a compact set $Q$ (with non-void interior) there
exists a constant $B(Q)$ such that one has:   For any $\varphi \in \DRd
= \ScRd \cap \CcRd$ (the space of infinitely smooth functions with compact support) with $\supp(\varphi) \subseteq x+Q$ for some $x \in \cG$: }
\begin{equation}\label{trlbdest1}
\tred{ |\sigma(\varphi)| \leq B(Q) \infnorm{\varphi} }.
\end{equation}
Equivalently one has: A tempered distribution defines a translation-bounded measure if and only if for any $p \in \DRd$ the family
$ (T_x p \cdot \sigma)$ constitutes a bounded family in $\MRdN$.
\end{lemma}
Dual to the atomic characterization  of $\WG$ we can also provide
a kind of {\it atomic representation} of
$\Wsp^*$, which works as follows:
\begin{lemma}\label{charWspdu}
Given any well-spread family $(x_i)_{i \in I}$ in $\cG$, the
elements $\sigma \in \Wsp^*$ can be characterized as the
$\wstd$convergent series of the form\footnote{Recall
that $(\sigma\cdot h) := \sigma(h \cdot f)$ by definition.}
\begin{equation}\label{sigWspdu}
   \sigma  = \sumiI \mu_i \cdot T_{x_i} p,
\end{equation}
for some fixed, non-zero $p \in \CcG$
and some bounded family $(\mu_i)_{i \in I}$   in $\MGN$.
\end{lemma}
\def\pxi{{T_{x_i}p}}
\begin{proof}
The proof has two directions.
First of all we fix some $U$-BUPU $\Psifam$ and some
$p \in \CcG$ with $\linfnorm{p} = 1$,
$ p(x) \equiv 1$ on $U$, and
hence $\psii = \psii \cdot \pxi$ for $i \in I$.

This allows us to decompose any linear functional $\sigma \in \Wsp^*$ in the usual way as a $\wstd$convergent series of the form
\begin{equation}\label{sigWstdec}
\sigma = \sumiI \sigma \cdot \psii
  = \sumiI (\sigma \cdot \psii)\cdot \pxi.
\end{equation}
We will check that the functionals
$\mu_i := \sigma \cdot \psii $
define a bounded family in $\MGN$.  In fact, we have,
 thanks to the atomic  characterization
\begin{equation} \label{muiestim1}
\Mbnorm{\mu_i} =
\Mbnorm{\sigma \psii} \leq  \linfnorm{\psii} \Mbnorm{\sigma \cdot (\pxi)}   \leq C \normta \sigma {\Wsp^*}, \quad i \in I.
\end{equation}

In order to prove the converse let $(\mu_i)_{i \in I}$ be a bounded family
in $\MGN$. We have to control the norm of the functional $\sigma$
given by Equ. (\ref{sigWstdec}).
Due to the atomic characterization it is enough to present an estimate
for the atoms, i.e. a uniform estimate (with respect to the sup-norm)
for functions $f \in \CcG$ with $\supp(f) \subset z+Q$, for some $z \in \cG$.
The assumptions concerning the family  $(x_i)_{i \in I}$ then imply, that
one has for the compact set $K = \supp(p)$ the following
 a uniform bound ({\it independent of $z$}):
$$ \# F =  \# \{ i \in I,  (x_i + I) \cap (z+K) \neq \emptyset\}
\leq C(X) < \infty.
$$

  \noindent
Using $\supp(\pxi\cdot f) \subseteq \supp(\pxi) = x_i + K$
and $\linfnorm{\pxi \cdot f} \leq \linfnorm{f}$ we conclude
\begin{equation} \label{sigestim04}
 |\sigma (f)| \leq  \sumiI  |(\mu_i \cdot \pxi)(f)|
=  \sumiF |\mu_i(\pxi \cdot f)| 
\leq C(X) \sup_{i \in I} \Mbnorm{\mu_i}
\linfnorm{f}.
\end{equation}
%
\end{proof}

%

\begin{remark} \label{katzrem}
This definition appears to be different from the setting
chosen in Katznelson's book \citeX{ka76}, p.127.
He assumes only instead of condition (1) that one has
a continuous embedding $\BspN \hkr \Liloc(\Rdst)$.
But due to the translation invariance property (2) imposed
on the norm of $\BspN$ this implies immediately that one
has $\BspN \hkr \WLilin$, which is a closed subspace
of $\WMlinf$ (the usual characterization of the
dual of $\WCOlisp$ in the context of Wiener amalgam spaces).

Conversely one can show that the continuous shift property
implies that in the case that $\LiGN$ is defined in the
usual way with the help of Lebesgue integration combined
with the existence of a Haar measure on $\cG$, the continuous
shift property (3) in conjunction with (1) and (2) actually
imply that $\Bsp$ is contained in the subspace
$\WLilin \subset \WMlinf$.
\end{remark}

Equipped with these spaces, which can be described now for any LCA
group $\cG$ without the use of the Haar measure or structure theory
we can give a definition of a homogeneous Banach space on $\cG$ (HBSG).
\begin{definition} \label{HomBSPG}
A Banach space $\BspN$ is called a {\it homogeneous Banach space}
on a LCA groups $\cG$  (HBSG) if one has
\begin{enumerate}
\item $\BspN \hkr \Wsp^*$;
\item Translation is isometric on $\BspN$, i.e.
$$ \tblue{ \Bnorm{T_x f} = \Bnorm{f}, \quad \forall f \in \Bsp, x \in \cG;}$$
\item Translation is strongly continuous on $\BspN$, i.e.\
$$ \tblue{ \lim_{x \to 0} \Bnorm{T_x f -f } = 0, \quad \forall f \in \Bsp}.$$
\end{enumerate}
\end{definition}

The following lemma provides a connection between the different notions.
For simplicity we formulate the result for $\cG = \Rdst$, endowed with
the Lebesgue integral. It is valid for general LCA groups. 
\begin{lemma} \label{Katzhomog}
For any HSBG on  $\cG = \Rdst$  we have $\BspN \hkr \Liloc(\Rdst)$.
\end{lemma}
\begin{proof}
In the current situation the abstract results imply  that
$\BspN$ is an essential Banach module over $\LiRd$ with respect
to convolution. By means of the Cohen-Hewitt
factorization Theorem (see  \citeX{hero70})  any $f \in \Bsp$ can be written as $f = g \ast h$, for some $g \in \LiRd$ and
some $h \in \Bsp \subset \WMlinf(\Rdst)$. But
the convolution relations for Wiener amalgams established
in \citeX{fe83} imply (altogether):
\begin{equation}\label{HBSG01}
\Bsp = \LiRd \ast \Bsp \subset \LiG \ast \WMlinf(\Rdst) \subset \WLilin(\Rdst)
\subset \Liloc(\Rdst).
\end{equation}
\end{proof}





\section{Homogeneous Banach Spaces as essential $\Lisp$-modules}
\label{secHBS}

Let us start with the comment that   the so-called {\it regular
representation of a group} $\cG$, i.e.\ the mapping which assigns
to any $x \in \CG$ the (left) translation operator $T_x$\footnote{This
operator is denoted by $L_x$ in \citeX{re68}.} is of course one
of the most important cases for the application of the abstract
principle developed in Section \ref{secHBS}.

It is also clear that the general assumptions which allow to
invoke Thm.  \ref{conv-meas-homBR0} are satisfied for any homogeneous
Banach space on $\cG$. Since such Banach spaces usually contain
many functions from $\CcG$ and since in this case it is clear
that the abstract form of the convolution coincides with the
pointwise action as defined via the pairing of $\COG$ and $\MG$
it is justified to still call the mapping $\mu \rhobul f$ in this
case convolution and simply write $\mu \ast f$. In view of
density considerations it is possible to verify that - in case
there are different possible interpretations of the symbol ``$\ast$''
the result does not depend on the context. 

This {\it seemingly harmless}, but
nevertheless highly non-trivial use of this symbol in situations
which are generated by different technical considerations are well
 justified in all the cases which are considered
here. Occasionally a strict verification of such a claim has to be
undertaken. However,  unlike the  approach taken occasionally by experts in distribution theory we take  care for the ``existence of the convolution product'' at an
individual level\footnote{In such a situation even the associative law
may fail!} we emphasize module actions and bilinear pairings for Banach
spaces, which are obtained by extension of standard operations.

We thus can summarize our findings so far in the following theorem:
\begin{theorem} \label{MGhomogBR1}
Let $\BspN$ be a homogeneous Banach space of an LCA group $\cG$. Then
$\BspN$ is a Banach module over $\MGN$ with respect to convolution.
In fact, the action of $\mu$ on $f \in \Bsp$ is defined as the
limit of expressions of the form $\DPsimu \ast f$, in the norm of $\BspN$.
\end{theorem}
While it is enough to know the Riemann integral (on $\CcRd$) for the
case $\cG = \Rdst$ (or similar elementary LCA groups), we have to invoke
to the existence of the Haar measure on $\cG$, which is a translation
invariant linear functional on $\CcG$ (in fact on $\WCOliG$). It allows
to endow $\CcG$ with the $\Lisp$-norm, and establish that $\CcG$ with
this norm is a normed space. With a little bit of extra work one then
goes on to show that the bounded measure $\mu_k$ induced by $ k \in \CcG$
via the mapping $f \mapsto \intG f(x)k(x)dx$, or better
$f \mapsto H(f\cdot k)$\footnote{Here we write $H$ for the Haar functional,
i.e.\ the linear functional arising in the construction of the so-called
Haar measure (see e.g.\ \citeX{de02}).} is in fact an isometric
embedding from $\CcG$ into $\MGN$. Consequently it makes sense to define
the space $\LiGN$ simply as the closure of $\CcG$, more precisely of
$\{ \mu_k \suth  k \in \CcG \}$ in $\MGN$.

Continuing our efforts to develop the foundations of harmonic
analysis without the use of measure theory we have to establish
a few basic properties:
\begin{lemma}\label{Liprops}
\begin{enumerate}
\item  
$\LiGN$ is a Banach space;
\item  $\LiGN$ is a homogeneous Banach space;
\item  In fact, $\LiGN$ is a closed ideal in $\MGN$.
\item  $\LiG$ is $\wstd$dense in $\MGN$.
\end{enumerate}
\end{lemma}
\begin{proof}
By definition $\LiGN$ is a closed subspace subspace of $\MGN$
and hence complete, and thus a Banach space. The uniform continuity
of any $k \in \CcG$ implies that $\linfnorm{T_x k - k} \to 0$
for $x \to 0$. Due to (joint) compact support of all these
translates (for $x$ near $0$, resp.\ the neutral element $e \in \cG$)
one also has $\lim_{x \to 0} \linorm{T_x f- f} = 0$ for $f \in \LiG$
by approximation. Due to the continuous embedding
$ \WCOliGN \hkr \COGN$ it is clear that $\LiG$ is contained in
$\Wsp^*$, and thus the formal axioms for an HBSG are satisfied.

As a consequence of Thm. \ref{MGhomogBR1} it is also an $\MG$-module
with respect to convolution and thus a closed ideal in $\MGN$, once
it is verified that the external action of $\MGN$ on $\LiGN$
is compatible with the internal (e.g. obtained by a pointwise
definition of $f \ast g(x)$ for $f,g \in \CcG$, or using Lebesgue
integration).  Observe that the convolution of a compactly supported measure
$\mu \in \MGN$ with $k \in \CcG$ is a continuous function in $\CcG$
and thus, by taking limits, $\LiG$ is a closed ideal of $\MGN$.
The pointwise relation
$\mu \ast k(x) = \mu(T_x k\checkm)$  implies
\begin{equation} \label{suppconv02}
\supp(\mu \ast k) \subset \supp(\mu) + \supp(k).
\end{equation}
Since any measure $\mu \in \MGN$ can be approximated by finite sums
of the form $\sumiF \mu \psii$ (in the norm of $\MG$) the
obvious estimate
\begin{equation}\label{L1convest05}
\linorm{\mu \ast k} = \Mbnorm{ \mu \star \mu_k}
\leq \Mbnorm{\mu} \Mbnorm{\mu_k} = \Mbnorm{\mu} \linorm{k},
\quad \mu \in \MG, \, k \in \CcG,
\end{equation}
we see that $\LiGN$ is a closed ideal in $\MGN$.

In order to verify the $\wstd$density of $\LiG$ in
$\MG$ it is enough to check that it is possible
to find an approximation of Dirac-measures in the
$\wstd$sense, because this implies the possibility
to approximate discrete measures by elements of $\LiG$
(in fact by elements in $\CcG$)  by transitivity. 

In fact, given $h \in \COG$ and $x \in \cG$ the uniform
continuity of $h$ implies that $\delx$ can be approximated
well by non-negative functions $k \in \CcG$ with small
support $U$ centered around $x$. In fact,  assuming
$ \intG k(x)dx =1 $ (just a normalization)
it is easy to estimate the difference
\begin{equation}\label{Diracwst01}
|\delta_x(f) - \mu_k(f)| =
| 1\cdot f(x)-\intG f(y) g(y)dy|\leq \int_U |f(x)-f(y)||k(x)|dx < \veps.
\end{equation}
\end{proof}

Next we can also recall the definition of a Segal algebra:
\begin{definition}  \label{Segaldef03}
A Banach space $\BspN$, which is continuously and densely embedded
into $\LiGN$, and which is also a homogeneous Banach space, is called
a \tred{\it Segal algebra} (in Reiter's sense, see \citeX{re68,rest00}).
\end{definition}

Our knowledge so far implies immediately the following claim:
\begin{lemma} \label{SegalIdeal01}
Any Segal algebra $\BspN$ is a so-called {\it Banach ideal}
in $\LiGN$, i.e.\ it is a Banach space with its own right,
and an (left) ideal in $\LiGN$, satisfying the estimate
\begin{equation}\label{LiSegEst1}
\Bnorm{g \ast f} \leq \linorm{g} \Bnorm{f}, \quad g \in \LiG, f \in \Bsp.
\end{equation}
\end{lemma}

In order to check that any homogeneous Banach space is an
{\it essential}  Banach module over $\LiGN$ we will prove  our
third main result of this article. We start from the same situation
as in Thm. \ref{conv-meas-homBR0}. The following theorem is inspired by the
results in \citeX{fe77-2}, in particular Thm.2.2 there, where
such result have been shown by different arguments.
\begin{theorem}\label{wstconvB} Given a HBSG and
a bounded and tight net $(\mu_\alpha)_{\alpha \in I}$ in $\MGN$
with
$$ \mu_0 = \wstlimal \mual $$
then one has norm convergence
$$ \limal \Bnorm{  \mual \ast f - \mu_0 \ast f} = 0, \quad f \in \Bsp.$$
\end{theorem}

The result will be realized in the abstract setting of Thm. \ref{conv-meas-homBR0}.
This is our third main result. It shows that in the current context
for bounded and tight nets in $\MGN$ $\wst$-convergence of measures
is giving strong operator norm convergence of the corresponding
convolution operators. 
\begin{theorem} \label{wsttightconv}
Let $\rho$ be a strongly continuous, isometric representation of
the locally compact group $\cG$ on the Banach space $\BspN$ and
$(\mu_\alpha)_{\alpha \in I}$  a bounded and tight net in $\MGN$
with $ \mu_0 = \wstlimal \mual $. Then one has
\begin{equation}\label{wstconv01}  \tred{
\limal \Bnorm{ \mual \rhobul f - \mu_0 \rhobul f }  = 0, \quad
\forall f \in \Bsp. }
\end{equation}
\end{theorem}
\begin{proof}
Given $\epso$ and $f \in \Bsp$ we have to find $\alpha_0$ such that
$\alpha \succeq \alpha_0$ implies
\begin{equation} \label{targetconv01}
\Bnorm{ \mual \rhobul f - \mu_0 \rhobul f } \leq \veps.
\end{equation}
For convenience we assume that $\Bnorm{f}$ = 1.

By Thm.\ref{conv-meas-homBR0} we can find $\UUG$ such that
for any $U-BUPU \, \Psi$ one has
\begin{equation} \label{discappr}
\Bnorm{ (\DPsi \mual - \mual) \rhobul f}  =
\Bnorm{\mual \rhobul f - \DPsi\mual \rhobul f} \leq \veps/4
\quad \forall  \alpha \in \{I,0\}.
\end{equation}
Let us now fix one such BUPU $\Psifam$. 
By the (definition) of tightness we find that there
exist compactly supported functions
$p \in \CcG$\footnote{One should think of plateau like
functions, as they arise e.g.\ as {\it finite} partial sums of the
form $\sumiF \psii$ from any BUPU $\Psi$ on $\cG$.}
such that
\begin{equation}\label{tightuse01}
\Mbnorm{ \mual \cdot p - \mual} \leq \veps/4, \quad
\forall \alpha \in I.
\end{equation}
By taking limits the estimate (\ref{tightuse01})
will be also valid for $\alpha = 0$ (the limit measure $\mu_0$).

Thus, up to a controllable error we may  assume
that the measures $\mualiI$ (and their limit $\mu_0$) have joint
compact support, 
and consequently there exists some finite set $F \subset I$ such that
$\mual(\psii) = 0$ for $i \in I \setminus F$, for all $\alinI$ and
$\alpha = 0$.

By the assumed $\wstd$convergence of the net $\mualiI$ we can find
some index $\alpha_0$ with
\begin{equation}\label{DPsiconv05}
\sumiF | \mu_\alpha(\psii) - \mu_0(\psii)|
\leq \veps/4, \quad \forall \alpha  \succeq \alpha_0,
\end{equation}
which in turn implies that we have for $\alpha \succeq \alpha_0$:
\begin{equation}\label{DPsiconv06}
\Bnorm{ \DPsi \mual \rhobul f -  \DPsi \mu_0 \rhobul f }
\leq \sumiF |\mu_\alpha(\psii)-\mu_0(\psii)|\Bnorm{\rho(x_i)f}
\leq \veps/4.
\end{equation}
By combining the estimates (\ref{discappr}),  (\ref{DPsiconv05})
and (\ref{DPsiconv06}) we have verified (\ref{targetconv01}), i.e.\
we can estimate \tblue{$ \Bnorm{(\mual - \mu_0) \rhobul f} $
} in the following way:
\[
\tblue{
\leq \Bnorm{ (\mual  - \DPsi \mual) \rhobul f} +
\Bnorm{(\DPsi\mual - \DPsi \mu_0)\rhobul f} + \newline
\Bnorm{(\DPsi \mu_0  - \mu_0) \rhobul f} \leq 3\veps/4.
}
\] 
\end{proof}

There are many applications of this rather strong statement, so that
we present only a striking one. As it is well known,  bounded approximate
units in $\LiGN$ are obtained by taking a sequence (or net) of
non-negative (for simplicity) functions $(k_\alpha)_{\alinI}$
with shrinking support and with $\intG k_\alpha(x) =1$ for all $\alinI$.
Such a sequence is often called a {\it Dirac sequence} in the literature,
and it is obviously tight and bounded in $\LiGN$.
It is a simple  exercise to verify that $\mual := \mu_{k_\alpha}$
is then a $\wstd$convergent net with
\begin{equation} \label{Diracsequ1}
\wstlimal \mual = \delta_0.
\end{equation}
As a consequence we thus have:
\begin{corollary} \label{BAILiconv1}
Given the situation of Thm. \ref{wstconvB}, and a bounded
approximate unit $(g_\alpha)_{\alinI}$  in $\LiGN$. Then
one has
\begin{equation}\label{BAIHBGS1}
\tred{ \limal \Bnorm{ g_\alpha \rhobul f - f} = 0,
\quad \forall f \in \Bsp.}
\end{equation}
\end{corollary}
As pointed out in Section 5.2 of \cite{fe17}  the net $\DPsimu$ is
providing a tight $\wstd$approximation to  $\mu$. Combining this fact
with the iteration principle (see \cite{ke75}, p.69) for convergent nets,
we come up with a verification of the associativity law which is required
for Banach modules.
\begin{corollary}\label{iterconv}
In the situation of Thm. \ref{wsttightconv} let $\mu_1, \mu_2 \in \MG$ be given. Then  %
\begin{equation}\label{Dpsicompo}
  \lim_{\sPsitoz} \Bnorm{ \DPsi\mu_1 \rhobul (\DPsi\mu_2 \rhobul f) -
  (\mu_1 \star \mu_2) \rhobul f } = 0, \quad f \in \Bsp.
\end{equation}
\end{corollary}

Combining the observations made so far  we come to
 the following final result,
 which shows that the notion of the \tblue{\it integrated
group representation} arises as a consequence of the approach presented
in this paper:
\begin{theorem} \label{HBGessLi1}
Given an isometric, strongly continuous representation of a locally
compact group $\cG$ on a Banach space $\BspN$, the restriction of
the Banach module action of $\MGN$ to the closed ideal $\LiGN$
turns $\BspN$ into an essential Banach module over $\LiGN$.

Conversely, the $\wstd$continuity of the action of $\MG$ for
bounded and tight families implies that the action is all of
$\MG$ is uniquely determined by the {\it integrated group action},
i.e. the $\LiG$-module properties.
\end{theorem}

\begin{remark}
We think that it is easier to obtain the integrated group representation
of $\LiGN$ on $\BspN$ by way of restriction, instead of going the more
cumbersome way of extending the representation of $\LiGN$ by taking
(vague) limits.
\end{remark}

\section{Some Basic Functional Analysis}   \label{secFA}

An important tool from functional analysis was the fact
that any Banach space is complete with respect to convergence
of \tred{\it Cauchy-net}, not just {\it Cauchy sequences.}

Although Cauchy-nets are (implicitly) appearing in many places,
e.g.\ in the definition of the Riemann integral, they are typically
not discussed as such. The reader could consult  BOURBAKI (\cite{bo04-5})
 for details on nets,  or
\citeX{me98-1} (Prop. 2.1.40), but in
order to make this note  more self-contained let us collect some
relevant facts.

\begin{definition} \label{orient01}
A set $(I,\succeq)$ is called a \tblue{\it directed set} with
respect to the {\it orientation} (given by $\succeq$), if it
satisfies the following properties:
\begin{enumerate}
\item one has transitivity, i.e.\ if
$ \alpha \succeq \beta$ and $\beta \succeq \gamma$ then
$ \alpha \succeq \gamma$;
\item Given $\alpha, \beta \in I$ there exists $\gamma \in I$
such that $\gamma \succeq \alpha$ and $\gamma \succeq \beta$.
\end{enumerate}
\end{definition}
Of course, in many cases one can have a partially ordered set
and choose $\gamma = \max(\alpha,\beta)$ in the above setting,
but this operation need not be meaningful in the general case.

\begin{definition} \label{netdef01}
A {\it net} in a set $X$ is a mapping from a directed set $(I,\succeq)$,
usually described as an indexed family $(x_\alpha)_{\alpha \in I}$.

A net in a metric space $(X,d)$ is called {\it convergent} if there
exists some $x_0 \in X$ such that one has:  Given $\epso$ there exists $\alpha_0$ such that
$$ \alpha \succeq \alpha_0 \Rightarrow  d(x_0,x_\alpha) < \veps. $$
In this case we also write:
$\lim_{\alpha \to \infty} x_\alpha = x_0.$
\end{definition}

Nets are natural generalizations  of sequences (and are thus
often just called {\it generalized sequences}). The analogue
of a Cauchy-sequence is of course a \tblue{\it Cauchy-net}.
\begin{definition} \label{Cauchynet02}
A  net $(x_\alpha)_{\alpha \in I}$ is a {\it Cauchy net} if
for any $\epso$   $\exists \, \alpha_0 \in I$ such that
$$ \alpha,\beta \succeq \alpha_0 \quad \Rightarrow \quad d(x_\alpha,x_\beta) < \veps. $$
\end{definition}

\begin{remark}
Typical nets relevant for our discussion are the nets of the form
$(\Strho g)_{\rho \to 0}$, with $ [\Strho g](x) = \rho^{-d}g(x/\rho)$,
with $\rho_1 \succeq \rho_2$ if $\rho_2 \leq \rho_1$, which are used
to generate {\it Dirac nets} (bounded approximate units) in $\LiRdN$.

Other nets occur naturally is the index set to the Riemannian sums
for an integral of the form $\int_a^b f(x)dx$, given by some finite decomposition of the interval $[a,b]$ and the choice of a family of
points $(\xi_i)$ in the corresponding intervals. As we all know
a Riemannian sum is considered good if the maximal length appearing
in the corresponding decomposition is controlled by a positive value
$\delo$. Furthermore, given two decomposition one can generate the
{\it joint refinement} as a decomposition which is ``better'' than
both of the decompositions generating it.
\end{remark}

\begin{theorem}  \label{Bancompl01}
A normed space  $\BspN$ is complete if and only if any Cauchy net is convergent
in $\BspN$.
\end{theorem}

Note: it is well known that a Banach spaces is complete if and only if
every Cauchy sequence is convergent, or equivalently, if every absolutely
convergent series is convergent in $\BspN$. It is also clear, that any
Cauchy sequence is a Cauchy net (using the index set $\Nst$ with natural
ordering as index set). Thus it is clear that we only have to verify
that any Cauchy sequence $\xalI$ is convergent in $\BspN$. 
\begin{proof}
First we determine a sequence $\varepsilon_n$, e.g. $\varepsilon = 2^{-n}$
for $n \geq 1$, and --following the definition of a Cauchy net -- a sequence
$ \alpha_n$ such that $ \alpha,\beta \succeq \alpha_n$
\begin{equation}\label{alphan1}
\alpha, \beta \succeq  \alpha_n  \quad \Rightarrow \quad
\Bnorm{x_\alpha - x_\beta} < \varepsilon_n.
\end{equation}
Without loss of generality (due to the majorization property) we can
determine the sequence $\alpha_n$ inductively with $\alpha_{n+1} \succeq \alpha_n$. Formally we choose $x_{\alpha_0} = 0 \in \Bsp$.

The series $ \sum_{n \geq 1} \left( x_{\alpha_n} - x_{\alpha_{n-1}} \right ) $
is then absolutely convergent, because 
$$ \sum_{n \geq 1}  \Bnorm{ x_{\alpha_n} - x_{\alpha_{n-1}}} \leq
\sum_{n \geq 1} \varepsilon_n \leq 1 < \infty. $$
Hence the partial sums are 
$$  \sum_{n = 1}^N \left( x_{\alpha_n} - x_{\alpha_{n-1}} \right )  = x_{\alpha_N} \quad (!) $$
are convergent, i.e.\ there exists some $x_0 \in \Bsp$ with
$$ \lim_{n \to \infty} x_{\alpha_n} = x_0 \,\,\, \mbox{in} \,\, \BspN. $$
Invoking the initial Cauchy-net condition we complete the argument
by showing (once a limit has been identified) that we  have indeed 
$$ \tred{\limalinf \,\, x_\alpha = x_0 \,\, \mbox{in} \,\,  \BspN.}$$
\end{proof}

\begin{remark}
It should be noted as a delicate point that the convergent Cauchy-sequence
obtained in the proof does not have to be a {\it subnet} of the original
Cauchy-net, because the notation of a subnet (which we do not need here)
is more complex than just the idea of a subsequence of a given sequence.
At least, it does not just mean reducing the index set (which for sequences
has a natural order) to a subset of the original index set with strictly
increasing enumeration of the elements of the subsequence.
\end{remark}

{\tt Acknowledgement:}  During the preparation of this paper the author
has received support from the FWF project [I 3403] and the ANACRES
network funded by OEAD.

\bibliographystyle{abbrv}

\end{document}